\documentclass[final]{article}
\usepackage{amsmath,amstext,amsfonts,amsthm,cite}
\usepackage[utf8]{inputenc}
\usepackage{times}
\usepackage{xcolor,framed,hyperref}
\usepackage[notref,notcite]{showkeys}
\usepackage{float}
\title{Multigrid methods for $\Hdiv$-conforming discontinuous
  Galerkin methods for the Stokes equations}
\author{G.~Kanschat\thanks{\texttt{kanschat@uni-heidelberg.de}, Interdisziplinäres Zentrum für Wissenschaftliches Rechnen
  (IWR), Universität Heidelberg, Im Neuenheimer Feld 368, 69120
  Heidelberg, Germany}
\and
Y.~Mao\thanks{\texttt{youlimao@tamu.edu}, Department of Mathematics, Texas A\&M
  University, 3368 TAMU, College Station, TX 77843, USA}}

\let\epsilon\varepsilon
\def\naught#1{#1^0}
\def\A{\mathcal A}
\def\B{\mathcal B}
\def\L{\mathcal L}
\def\I{\mathcal I}

\def\F{\mathbb F}
\def\T{\mathbb T}
\def\R{\mathbb R}

\def\d{\partial}
\def\n{\mathbf n}
\def\grad{\nabla}
\def\div{\nabla\!\cdot\!}

\def\ntrace#1{#1\!\cdot\!\n}
\newcommand{\mvl}[1]{\{\!\!\{#1\}\!\!\}}             
\def\rf#1{\widehat{#1}}

\def\Hdiv{H^{\text{div}}}                            

\newcommand{\norm}[1]{\bigl\|#1\bigr\|}

\def\form(#1,#2){\left(#1,#2\right)}
\def\forme(#1,#2){\bigl<#1,#2\bigr>}
\newtheorem{theorem}{Theorem}
\newtheorem{lemma}{Lemma}[theorem]
\newtheorem{proposition}{Proposition}[theorem]

 {\endMakeFramed}
\begin{document}

\maketitle

\abstract{%
  A multigrid method for the Stokes system discretized with an
  $\Hdiv$-conforming discontinuous Galerkin method is
  presented. It acts on the combined velocity and pressure spaces and
  thus does not need a Schur complement approximation. The smoothers
  used are of overlapping Schwarz type and employ a local Helmholtz
  decomposition. Additionally, we use the fact that the discretization
  provides nested divergence free subspaces. We present the
  convergence analysis and numerical evidence that convergence rates
  are not only independent of mesh size, but also reasonably small.
}

\section{Introduction}

The efficient solution of the Stokes equations is an important step
in the development of fast flow solvers. In this paper we present analysis and
numerical results for a multigrid method with subspace correction
smoother, which performs very efficiently on divergence-conforming
discretizations with interior penalty. We obtain convergence rates for
the Stokes problem which are comparable to those for the Laplacian.

Multigrid methods are known to be the most efficient preconditioners
and solvers for diffusion problems. Nevertheless, for Stokes
equations, the divergence constraint makes the solution process more
complicated. A typical solution employs the use of block
preconditioners, e.~g.~\cite{ElmanSilvesterWathen02,KayLoghinWathen02,MurphyGolubWathen00,Kanschat05},
but their disadvantage is, that their
performance is limited by the inf-sup constant of the problem. This
could be avoided, if the multigrid method operated on the divergence
free subspace directly, and thus would not have to deal with the
saddle point problem at all. Such methods have been developed in
different context and have proven very successful as reported for
instance by Hiptmair~\cite{Hiptmair99multigrid} for Maxwell equations and
by Schöberl~\cite{Schoeberl99dissertation} for incompressible
elasticity with reduced integration.

The main ingredients into such a method are a smoother which operates
on the divergence free subspace and a grid transfer operator from
coarse to fine mesh which maps the coarse divergence free subspace
into the fine one. The second objective can be achieved by using a
mixed finite element discretization for which the weakly divergence
free functions are point-wise divergence free. For such a
discretization, the natural finite element embedding operator from
coarse to fine mesh does not increase the divergence of a
function. Discretizations of this type are available, such as
for instance in Scott and
Vogelius~\cite{ScottVogelius85,Vogelius83}, Neilan and
coauthors~\cite{GuzmanNeilan14,FalkNeilan13} and
Zhang~\cite{Zhang09tet,Zhang09rectangular}. Here, we focus on the
divergence conforming discontinuous Galerkin (DivDG) method of Cockburn,
Kanschat, and Schötzau~\cite{CockburnKanschatSchoetzau07} due to its
simplicity.

Following the approach by Schöberl~\cite{Schoeberl99dissertation},
in order to study smoothers for the Stokes equations, we first
consider a problem on the velocity space only with penalty for the
divergence. This leads to a singularly perturbed problem with an
operator with a large kernel.  When it comes to smoothers for such
operators, there are two basic options.  One approach is to smooth the
kernel space explicitly, as proposed for instance by
Hiptmair~\cite{Hiptmair99multigrid} and Xu in~\cite{HiptmairXu09}. The
other option was presented by Arnold, Falk, and Winther
in~\cite{ArnoldFalkWinther97Hdiv,ArnoldFalkWinther00} and smoothens the
kernel implicitly, while never employing an explicit description of it.

We follow the implicit approach and use the same domain decomposition
principle (i.e additive and multiplicative Schwarz methods and vertex
patches), but instead of the Maxwell or divergence dominated mass
matrix as in~\cite{ArnoldFalkWinther97Hdiv,ArnoldFalkWinther00} apply
it to the DivDG Stokes discretization. Then, we prove the convergence
of  the multigrid method with variable V-cycle algorithm for the
singularly perturbed problem.
The second pillar we rest on is the equivalence between
singularly perturbed, divergence dominated elliptic forms and mixed
formulations established by Schöberl
in~\cite{Schoeberl99dissertation,Schoeberl99}. This equivalence allows
us to apply the smoother to a mixed formulation of nearly
incompressible elasticity and then to proceed to the Stokes limit. As
far as we know, the combination of these techniques has not been
applied the DivDG method in~\cite{CockburnKanschatSchoetzau07}. Since our analysis
is based on domain decomposition, fundamental results are also drawn
from the seminal paper by Feng and
Karakashian~\cite{FengKarakashian01} on domain decomposition for
discontinuous Galerkin methods for elliptic problems.

There is a close relation between our technique and the smoother
suggested by Vanka in~\cite{Vanka86} for the MAC scheme: the
MAC scheme can be considered the lowest order case of the DivDG
methods (see~\cite{Kanschat08mac}). In this case, the subspace decomposition
structure of Vanka smoother corresponds to Neumann problems on cells,
while our smoother is based on Dirichlet problems for vertex patches.
Generalizations of the Vanka smoother have been applied
successfully to different other discretizations albeit their
velocity-pressure spaces are not matched in the sense
of~\eqref{eq:divergence-relation} (see for
instance\cite{Turek93,WobkerTurek09} and literature cited there).

Recently, an alternative preconditioning method for Stokes
discretizations of the same type as here has been introduced
in~\cite{AyusoBrezziMariniXuZikatanov13} by Ayuso et al. Their method
is based on auxiliary spaces introduced by Hiptmair and Xu
in~\cite{HiptmairXu09}.
The exact
sequence property of the divergence-conforming velocity element plays
a crucial role as in our scheme, but their preconditioner uses a multigrid method for
the biharmonic problem to solve the Stokes problem. As a consequence,
it is not possible to use the preconditioning method for no-slip
boundary conditions. On the other hand, it has been demonstrated
in~\cite{KanschatSharma14} that the multigrid method here can be lifted
to the biharmonic problem, providing an efficient method for clamped
boundary conditions.

The paper is organized as follows. In Section~\ref{sec:discretization}
we present the model problem and the DG discretization. The multigrid
method and domain decomposition smoother are derived in
Section~\ref{sec:multigrid}.  Section~\ref{sec:analysis} is devoted to
the convergence analysis of our preconditioning technique with the man
result in Theorem~\ref{theorem:main} on
page~\pageref{theorem:main}. The paper concludes with numerical
experiments in Section~\ref{sec:experiments}.

\section{The Stokes problem and its discretization}
\label{sec:discretization}

We consider discretizations of the Stokes equations
\begin{gather}
  \label{eq:Stokes}
  \arraycolsep2pt
  \begin{array}{rccclcll}
    -\triangle u &+& \nabla p &=& f &\quad& \text{in}& \Omega,
    \\
    \div u && &=& 0 && \text{in}& \Omega,
    \\
    && u &=& u^B &&\text{on }& \partial\Omega,
  \end{array}
\end{gather}
with no-slip boundary conditions on a bounded and convex domain
$\Omega\subset \R^d$ with dimension $d=2,3$. The natural solution
spaces for this problem are $V = H^1_0(\Omega;\R^d)$ for the velocity
$u$ and the space of mean value free square integrable functions $Q =
L^2_0(\Omega)$ for the pressure $p$, although we point out that other
well-posed boundary conditions do not pose a problem.

In order to obtain a finite element discretization, we partition the
domain $\Omega$ into a hierarchy of meshes
$\{\T_\ell\}_{\ell=0,\dots,L}$ of parallelogram and parallelepiped
cells in two and three dimensions, respectively. In view of multilevel
methods, the index $\ell$ refers to the mesh level defined as follows:
let a coarse mesh $\T_0$ be given. The mesh hierarchy is defined
recursively, such that the cells of $\T_{\ell+1}$ are obtained by
splitting each cell of $\T_\ell$ into $2^d$ congruent children
(refinement).  These meshes are nested in the sense that every cell of
$\T_\ell$ is equal to the union of its four children. We define the
mesh size $h_\ell$ as the maximum of the diameters of the cells of
$\T_\ell$. Due to the refinement process, we have $h_\ell = 2^{-\ell} h_0$.

By construction, these meshes are conforming in the sense that every
face of a cell is either at the boundary or a whole face of another
cell; nevertheless, local refinement and hanging nodes do not pose a
particular problem, since they can be treated
following~\cite{JanssenKanschat11,Kanschat04}. By
$\F_\ell$ we denote the set of all faces of the mesh $\T_\ell$, which
is composed of the set of interior faces $\F_\ell^i$ and the set of
all boundary faces $\F_\ell^\d$.

We introduce a short hand notation for integral forms on $\T_\ell$ and
on $\F_\ell$ by
\begin{xalignat*}{2}
  \form(\phi,\psi)_{\T_\ell}
  &= \sum_{T \in \T_\ell}\int_T \phi \odot\psi \,dx,
  &
  \forme(\phi,\psi)_{\F_\ell}
  &= \sum_{F\in \F_\ell}\int_F \phi \odot\psi  \, ds,
  \\
  \norm{\phi}_{\T_\ell}
  &= \biggl(\sum_{T \in \T_\ell}
  \int_T |\phi|^2\,dx\biggr)^{\frac{1}{2}},
  &
  \norm{\phi}_{\F_\ell}
  &= \biggl(\sum_{F\in \F_\ell}\int_F |\phi|^2 \, ds\biggr)^{\frac{1}{2}},
\end{xalignat*}
The point-wise multiplication operator $\phi\odot\psi$ refers to the product $\phi\psi$,
the scalar product $\phi\cdot\psi$ and the double contraction
$\phi:\psi$ for scalar, vector and tensor arguments, respectively. The
modulus $|\phi| = \sqrt{\phi\odot\phi}$ is defined accordingly.

In order to discretize~\eqref{eq:Stokes} on the mesh $\T_\ell$, we
choose discrete subspaces $X_\ell = V_\ell \times Q_\ell$, where
$Q_\ell \subset Q$.
Following\cite{CockburnKanschatSchoetzau07}, we employ discrete
subspaces $V_\ell$ of the space $\Hdiv_0(\Omega)$, where
\begin{align*}
  \Hdiv(\Omega) &= \bigl\{v\in L^2(\Omega;\R^d)
  \big| \div v \in L^2(\Omega)
  \bigr\},
  \\
  \Hdiv_0(\Omega) &= \bigl\{v\in \Hdiv(\Omega)
  \big| \ntrace{v} = 0 \quad\text{on } \d\Omega
  \bigr\}.
\end{align*}
Here, we choose the well-known Raviart--Thomas
space\cite{RaviartThomas77}, but we point out that any pair of
velocity spaces $V_\ell$ and pressure spaces $Q_\ell$ is admissible,
if the key relation
\begin{gather}
  \label{eq:divergence-relation}
  \div V_\ell = Q_\ell
\end{gather}
holds. The details of constructing the Raviart--Thomas space follow.

Each cell $T\in \T_\ell$ can be obtained as the image of
a linear mapping $\Psi_T$ of the reference cell $\rf T = [0,1]^d$. On the reference cell, we define two polynomial spaces: first, $\rf
Q_k$, the space of polynomials in $d$ variables, such that the degree
with respect to each variable does not exceed $k$. Second, we consider
the vector valued space of Raviart--Thomas polynomials $\rf V_k = \rf
Q_k^d + x\rf Q_k$. Polynomial spaces $V_T$ and $Q_T$ on the mesh cell
$T$ are obtained by the pull-back under the mapping $\Psi_T$ (see for
instance~\cite{ArnoldFalkWinther10}). The polynomial degree $k$ is
arbitrary, but chosen uniformly on the whole mesh. Thus, we will omit
the index $k$ from now on. Concluding this construction, we obtain the
finite element spaces
\begin{align*}
  V_\ell &= \bigl\{v\in\Hdiv_0(\Omega)
  \big| \forall T\in \T_\ell : v_{|T} \in V_T \bigr\},
  \\
  Q_\ell &= \bigl\{ q\in L^2_0(\Omega)
  \big| \forall T\in \T_\ell : q_{|T} \in Q_T \bigr\}.
\end{align*}

\subsection{Discontinuous Galerkin discretization}
While the fact that $V_\ell$ is a subspace of $\Hdiv_0(\Omega)$
implies continuity of the normal component of its functions across
interfaces between cells, this is not true for tangential
components. Thus, $V_\ell \not\subset H^1(\Omega;\R^d)$, and it cannot
be used immediately to discretize~\eqref{eq:Stokes}. We follow the
example in for
instance~\cite{CockburnKanschatSchoetzau07,KanschatSchoetzau08,KanschatRiviere10}
and apply a DG formulation to the discretization of the elliptic
operator. Here, we focus on the interior penalty
method\cite{Arnold82,Nitsche71}. Let $T_1$ and $T_2$ be two mesh cells
with a joint face $F$, and let $u_1$ and $u_2$ be the traces of a
function $u$ on
$F$ from $T_1$ and $T_2$, respectively. On this face $F$, we introduce
the averaging operator
\begin{gather}
  \mvl{u} = \frac{u_1 + u_2}{2}.
\end{gather}
In this notation, the interior penalty bilinear form reads
\begin{gather}
  \begin{split}
    \label{eq:IP}
    a_\ell(u,v) =& \form(\nabla u, \nabla v)_{\T_\ell}
    + 4\forme(\sigma_L\mvl{u\otimes\n},\mvl{v\otimes\n})_{\F_\ell^i}
    \\
    &- 2\forme(\mvl{\nabla u},\mvl{\n\otimes v})_{\F_\ell^i}
    - 2\forme(\mvl{\nabla v},\mvl{\n\otimes u})_{\F_\ell^i}
    \\
    &
    + 2\forme(\sigma_L u,v)_{\F_\ell^\d}
    - \forme(\d_n u,v)_{\F_\ell^\d}
    - \forme(\d_n v,u)_{\F_\ell^\d}.
  \end{split}
\end{gather}
The operator ``$\otimes$'' denotes the Kronecker product of two
vectors. We note that the term $4\mvl{u\otimes\n}:\mvl{v\otimes\n}$
actually denotes the product of the jumps of $u$ and $v$.

The discrete weak formulation of \eqref{eq:Stokes} reads now: find $(u_\ell,p_\ell)
\in V_\ell\times Q_\ell$, such that for all test functions $v_\ell \in
V_\ell$ and $q_\ell \in Q_\ell$ there holds
\begin{gather}
  \label{eq:Stokes-weak}
  \A_\ell\left(
    \begin{pmatrix}
      u_\ell\\p_\ell
    \end{pmatrix}
    ,
    \begin{pmatrix}
      v_\ell\\q_\ell
    \end{pmatrix}
\right)
  \equiv
  a_\ell(u_\ell,v_\ell)
  + \form(p_\ell, \div v_\ell)
  - \form(q_\ell, \div u_\ell)
  = \mathcal{F}(v_\ell, q_\ell) \equiv\form(f,v_\ell).
\end{gather}

Discussion on the existence and uniqueness of such solutions can be
found for instance in
~\cite{CockburnKanschatSchoetzau07,CockburnKanschatSchoetzauSchwab02,HansboLarson02,KanschatRiviere10}.
Here, we summarize, that is symmetric.  If $\sigma_L$ is sufficiently
large, the form $a_\ell(.,.)$ is positive definite independently of the
multigrid level $\ell \in [0,L]$, and that thus we can define a norm
on $V_\ell$ by
\begin{gather}
  \norm{v_\ell}_{V_\ell} = \sqrt{a_\ell(v_\ell,v_\ell)}.
\end{gather}
In order to obtain optimal convergence results and to satisfy
Proposition~\ref{proposition:estimate-3} below, $\sigma_L$ is chosen
as $\overline\sigma/h_{L}$, where $h_L$ is mesh size on the finest
level $L$ and $\overline\sigma$ is a positive constant depending on
the polynomial degree. By this choice, the bilinear
forms on lower levels are inherited from finer levels in the sense,
that
\begin{gather}
  \label{eq:5}
  a_\ell(u_\ell, v_\ell) = a_L(u_\ell,v_\ell), \quad \forall\; u_\ell,
  v_\ell \in V_\ell.
\end{gather}

 A particular feature of this method is
(see~\cite{CockburnKanschatSchoetzau05,CockburnKanschatSchoetzau07}),
that the solution $u_\ell$ is in the divergence free subspace
\begin{gather}
  V_\ell^0
  = \bigl\{ v_\ell \in V_\ell \big| \div v_\ell = 0 \bigr\},
\end{gather}
where the divergence condition is to be understood in the strong sense.

\begin{proposition}[Inf-sup condition]
  For any pressure function $q\in Q_\ell$, there exists a velocity
  function $v\in V_\ell$, satisfying
  \begin{gather}
    \label{eq:inf-sup}
    \inf_{q\in Q_\ell}
    \sup_{v\in V_\ell}
    \frac{\form(q,\div v)}{\norm{v}_{V_\ell} \norm{q}_{Q_\ell}}
    \geq  \gamma_\ell > 0 
  \end{gather}
 where $\gamma_\ell = c\sqrt{\frac{h_L}{h_\ell}} = c\sqrt{2^{\ell-L}}$
 and $c$ is a constant independent of the multigrid level $\ell$.
\end{proposition}

\begin{proof}
  The proof of this proposition can be found in~\cite[Section
  6.4]{SchoetzauSchwabToselli03}. Indeed, a different result is proven
  there, with $\gamma_\ell \approx 1/k$, where $k$ is the
  polynomial degree in the $hp$-method. Thorough study of the proof
  though reveals, that this $k$-dependence is due to the penalty
  parameter of the form $\sigma_\ell \approx k^2/h_\ell$. In our case, the
  penalty parameter depends on the fine mesh, not on $h_\ell$, such
  that $\sigma_\ell \approx (h_\ell/h_L)/h_\ell$, and that the role
  of the $k^2$ in the penalty is taken by the factor $h_\ell/h_L$.
\end{proof}

For any $u\in V_\ell$, we consider the following discrete Helmholtz
decomposition:
\begin{gather}
  \label{eq:Helmholtz decomposition}
  u = \naught{u} + u^{\perp}
\end{gather}
where $\naught{u} \in V_{\ell}^{0}$ is the divergence free
part and $u^{\perp}$ belongs to its
$a_\ell(.,.)$-orthogonal complement. For functions from this
complement holds the estimate:
\begin{lemma}
  \label{lemma:estimate-1}
  Let $u^{\perp}\in V_\ell$ be $a_\ell(.,.)$-orthogonal to $V_\ell^0$, that is,
  \begin{gather*}
    a_\ell(u^{\perp}, v) = 0 \quad \forall \;v\in V_\ell^0.
  \end{gather*}
  Then, there is a constant $\alpha>0$ such that
  \begin{gather}
    \label{eq:norm-estimate}
    \frac{\alpha}{d^2}\norm{\div u^{\perp}}^{2}
    \leq
    a_\ell(u^{\perp}, u^{\perp}) 
    \leq
    \frac1\gamma_\ell \norm{\div u^{\perp}}^{2},
  \end{gather}
  $\gamma_\ell$ is the inf-sup constant from inequality~\eqref{eq:inf-sup}.
\end{lemma}

\begin{proof}
  On the left side, we already argued above that $\sigma_L$ is chosen
  large enough such that $a_\ell(.,.)$ is uniformly positive
  definite. Thus, we have with a positive constant $\alpha$
  \begin{gather*}
    \alpha\|\nabla u^\perp\|_{\T_\ell}^2 \le a_\ell(u^{\perp},u^{\perp}).
  \end{gather*}
  But then,
  \begin{gather*}
    \form(\div u^{\perp}, \div u^{\perp})_{\Omega}
    \leq d^2\form (\grad u^{\perp},\grad u^{\perp} )_{\T_\ell} \le
     \frac{d^2}{\alpha} a_\ell(u^{\perp},u^{\perp}),
  \end{gather*}
  On the right side, let $q=\div u^{\perp}$.  Then $q\in Q_\ell$ due
  to~\eqref{eq:divergence-relation}. From~\eqref{eq:inf-sup}, we
  conclude that there is $u \in V_\ell$ such that $\div u = q$ and
  $\norm{u}_{V_\ell} \leq 1/\gamma_\ell \|q\|$.  On the other hand,
  $u^{\perp}$ is the error of the orthogonal projection into
  $V_\ell^0$. Thus, $u^{\perp}$ must be the element with minimal norm,
  and in particular $\norm{u^{\perp}}_{V_\ell}\le \norm{u}_{V_\ell}$.
\end{proof}

\subsection{The nearly incompressible problem}

We are going to prove convergence uniform with respect to the refinement level
$\ell$ of the proposed multigrid method for the Stokes
problem by deviating twice. First, we provide estimates robust with
respect to the parameter $\epsilon$
of the nearly incompressible problem: find $(u_\ell, p_\ell) \in
V_\ell \times Q_\ell$ such that for all $(v_\ell, q_\ell) \in V_\ell
\times Q_\ell$ there holds
\begin{gather}
  \label{eq:nearly-incompressible}
  \A_\ell\left(
    \begin{pmatrix}
      u_\ell\\p_\ell      
    \end{pmatrix}
    ,
    \begin{pmatrix}
      v_\ell\\q_\ell
    \end{pmatrix}
    \right)
    + \epsilon \form(p_\ell, q_\ell)
  = \mathcal{F}(v_\ell, q_\ell).
\end{gather}

This problem is connected with the simpler penalty bilinear form (see for
instance also~\cite{HansboLarson02})
\begin{gather}
  \label{eq:elliptic-1}
  A_{\ell,\epsilon}(u_\ell,v_\ell) \equiv a_{\ell}(u_\ell, v_\ell)
  + \epsilon^{-1} \form(\div u_\ell, \div  v_\ell)
\end{gather}
and the singularly perturbed, elliptic problem: find $u_\ell\in
V_\ell$ such that for all $v_\ell\in V_\ell$ there holds
\begin{gather}
  \label{eq:elliptic-2}
  A_{\ell,\epsilon}(u_\ell,v_\ell) = \form(f,v_\ell).
\end{gather}

\begin{lemma}
  \label{lemma:equivalence-1}
  Let $(u_m, p_m)$ be the solution to~\eqref{eq:nearly-incompressible}
  and $u_e$ be the solution to~\eqref{eq:elliptic-2}. Then,
  if~\eqref{eq:divergence-relation} holds, the following equations
  hold true:
  \begin{gather*}
    u_m = u_e, \qquad\text{and}\qquad \epsilon p_m = \div u_m = \div u_e.
  \end{gather*}
\end{lemma}

\begin{proof}
  Testing~\eqref{eq:nearly-incompressible} with $v_\ell=0$ and
  $q_\ell\in Q_\ell$ yields
  \begin{gather*}
    - \form(\div u_m, q_\ell) + \epsilon (p_m,q_\ell) = 0
    \quad\forall\; q_\ell\in Q_\ell.
  \end{gather*}
  Due to~\eqref{eq:divergence-relation}, this translates to the
  point-wise equality $\epsilon p_m = \div u_m$. Substituting $p_m$
  in~\eqref{eq:nearly-incompressible} and testing with the pair
  $(v_\ell, \div v_\ell)$, which is possible again due
  to~\eqref{eq:divergence-relation}, we obtain that $u_m$
  solves~\eqref{eq:elliptic-2}.
  
  If on the other hand $u_e$ solves~\eqref{eq:elliptic-2}, we
  introduce $p_e = \frac{1}{\epsilon}\div u_e$, which translates to
  \begin{gather*}
    - \form(\div u_e, q_\ell) + \epsilon (p_e,q_\ell) = 0
    \quad\forall\; q_\ell\in Q_\ell,
  \end{gather*}
  corresponding to~\eqref{eq:nearly-incompressible} tested with $(0,q_\ell)$.
  On the other hand,~\eqref{eq:nearly-incompressible} tested with
  $(v_\ell,0)$ is obtained directly from~\eqref{eq:elliptic-2}
  substituting $p_e$. Thus, the equivalence is proven.
\end{proof}

In order to help keeping the notation separate, we adopt the following
convention: the subscript $\epsilon$ is dropped wherever
possible. Furthermore, curly letters refer to the mixed form, while
straight capitals refer to operators on the velocity space only. Thus:
\begin{align*}
  a_\ell(u,v) &\text{ the vector valued interior penalty form} \\
  A_{\ell}(u,v) &\text{ the form of the
    singularly perturbed, elliptic problem~\eqref{eq:elliptic-2}}\\
  \A_\ell\left(
    \begin{pmatrix}
      u\\p
    \end{pmatrix}
    ,
    \begin{pmatrix}
      v\\q
    \end{pmatrix}
  \right) &\text{ the mixed bilinear form~\eqref{eq:nearly-incompressible}}
\end{align*}
Similarly, capital letters like in $R_\ell$ for the smoother~\eqref{eq:smoother-elliptic}
refer to the singularly perturbed, elliptic problem, while $\mathcal R_\ell$
is the corresponding symbol for the Stokes smoother~\eqref{eq:smoother-Stokes}.
Additionally, we associate operators with bilinear forms using the
same symbol:
\begin{xalignat*}{3}
  A_{\ell,\epsilon}&: V_\ell\to V_\ell & \form(A_{\ell,\epsilon} u,v)
  &= A_{\ell,\epsilon}\form(u,v) = A_{\ell}(u,v) = A_L(u,v) & \forall u,v &\in V_\ell
  \\
  \A_{\ell,\epsilon}&: X_\ell\to X_\ell & \form(\A_{\ell,\epsilon} x,y)
  &= \A_{\ell,\epsilon}\form(x,y) = \A_\ell(x,y) = \A_L(x,y) & \forall x,y &\in X_\ell
\end{xalignat*}

\section{Multigrid method}
\label{sec:multigrid}

In Section~\ref{sec:discretization}, we introduced hierarchies of
meshes $\{\T_\ell\}$. Due to the nestedness of mesh cells, the finite
element spaces associated with these meshes are nested as well:
\begin{gather*}
  \arraycolsep5pt
  \begin{matrix}
    V_0 &\subset& V_1 &\subset& \dots &\subset& V_L,
    \\
    Q_0 &\subset& Q_1 &\subset& \dots &\subset& Q_L.
    \\
    X_0 = V_0 \times Q_0 &\subset& X_1 &\subset& \dots
    &\subset& V_L \times Q_L = X_L. \\
  \end{matrix}
\end{gather*}
This relation also extends to the divergence free
subspaces, see for instance~\cite{KanschatSharma14}:
\begin{gather}
  \arraycolsep5pt
  \begin{matrix}
    V_0^0 &\subset& V_1^0 &\subset& \dots &\subset& V_L^0.
  \end{matrix}
\end{gather}

The nestedness of the spaces implies that there is a sequence of
natural injections $\I_\ell: X_\ell \to X_{\ell+1}$ of the form
$\I_\ell(v_\ell, q_\ell) = (I_{\ell,u}v_\ell, I_{\ell,p}q_\ell)$, such
that
\begin{alignat}{4}
  I_{\ell,u}&:\;& V_\ell &\to V_{\ell+1},
  &\qquad
  I_{\ell,p}&:\;& Q_\ell &\to Q_{\ell+1},
  \\
  I_{\ell,u}&:& V_\ell^0 &\to V_{\ell+1}^0.
\end{alignat}
The $L^2$-projection from $X_{\ell+1} \to X_\ell$ is defined by
$\I^t_\ell(v_\ell, q_\ell) = (I^t_{\ell,u}v_\ell,I^t_{\ell,p}q_\ell)$
with
\begin{xalignat}{2}
  \label{eq:3}
  \form(v_{\ell+1}-I^t_{\ell,u}v_{\ell+1}, w_\ell) &= 0 \; \forall w_\ell
  \in V_\ell 
  &
  \form(q_{\ell+1}-I^t_{\ell,p}q_{\ell+1}, r_\ell) &= 0 \; \forall r_\ell
  \in Q_\ell.
\end{xalignat}
The $\A$-orthogonal projection $\mathcal{P}_{\ell}$ from $(V_L \times
Q_L) \rightarrow (V_{\ell} \times Q_{\ell})$ is defined by
\begin{gather}
  \label{eq: Ritz projection}
  \A_L \form(\mathcal P_{\ell}
  \begin{pmatrix}
    u\\p
  \end{pmatrix}
  ,
  \begin{pmatrix}
    v_{\ell}\\q_{\ell}
  \end{pmatrix}
  )
  = \A_L \form(
  \begin{pmatrix}
    u\\p
  \end{pmatrix}
  ,
  \begin{pmatrix}
    v_{\ell}\\q_{\ell}
  \end{pmatrix}
  )
\end{gather}
for all $ (u,p) \in (V_L \times Q_L), (v_{\ell},q_{\ell}) \in V_{\ell}
\times Q_{\ell}$.  Similarly, The $A$-orthogonal projection $P_{\ell}$
from $V_L \rightarrow V_{\ell} $ is defined by
\begin{gather}
\label{eq: ritz projection}
  A_L( P_{\ell}u, v_{\ell})
  = A_L(u,v_{\ell})
\end{gather}
for all $ u \in V_L, v_{\ell} \in V_{\ell}$. 

\subsection{The V-cycle algorithm}
\label{sec:V-cycle}
In this subsection we define V-cycle multigrid preconditioners
$\B_{\ell,\epsilon}$ and $B_{\ell,\epsilon}$ for the operators
$\A_{\ell,\epsilon}$ and $A_{\ell,\epsilon}$, respectively. For
simplicity of the presentation, we drop the index $\epsilon$.

First, we define the action of the multigrid preconditioner
$\B_\ell : X_\ell \rightarrow X_\ell $ recursively as the
multigrid V-cycle with $m(\ell)\geq 1$ pre- and post-smoothing steps. Let
$\mathcal{R}_\ell$ be a suitable smoother. Let $\B_0 =
\A_0^{-1}$. For $\ell \ge 1$, define the action of $\B_\ell$ on a
vector $\L_\ell = (f_\ell,g_\ell)$ by

\begin{subequations}
\begin{enumerate}
\item Pre-smoothing: begin with $(u_0, p_0) = (0,0)$ and let
  \begin{gather}
    \begin{pmatrix}
      u_{i}\\p_{i}
    \end{pmatrix}
    =
    \begin{pmatrix}
      u_{i-1}\\p_{i-1}
    \end{pmatrix}
    + \mathcal{R}_{\ell}
    \left(\L_\ell - \A_\ell
      \begin{pmatrix}
        u_{i-1}\\p_{i-1}
      \end{pmatrix}
    \right)
    \quad i= 1, \dots, m(\ell),
  \end{gather}
  
\item Coarse grid correction:
  \begin{gather}
    \begin{pmatrix}
      u_{m(\ell)+1}\\p_{m(\ell)+1}
    \end{pmatrix}
    =
    \begin{pmatrix}
      u_{m(\ell)}\\p_{m(\ell)}
    \end{pmatrix}
    + \B_{\ell-1} \I_{\ell-1}^t \left(\L_\ell - \A_\ell
      \begin{pmatrix}
        u_{m(\ell)}\\p_{m(\ell)} 
      \end{pmatrix}
    \right),
  \end{gather}

\item Post-smoothing:
  \begin{gather}
    \begin{pmatrix}
      u_{i}\\p_{i}
    \end{pmatrix}
    =
    \begin{pmatrix}
      u_{i-1}\\p_{i-1}
    \end{pmatrix}
    + \mathcal{R}_{\ell}
    \left(\L_\ell - \A_\ell
      \begin{pmatrix}
        u_{i-1}\\p_{i-1}
      \end{pmatrix}
    \right),
    \quad i= m(\ell)+2, \dots, 2m(\ell)+1
  \end{gather}
\item Assign: 
  \begin{gather}
    \B_\ell \L_\ell =
    \begin{pmatrix}
      u_{2m(\ell)+1}\\p_{2m(\ell)+1}
    \end{pmatrix}
  \end{gather}
\end{enumerate}  
\end{subequations}

We distinguish between the standard and variable V-cycle algorithms by
the choice
\begin{gather*}
  m(\ell) =
  \begin{cases}
    m(L) &\text{standard V-cycle,} \\
     m(L) 2^{L-\ell} &\text{variable V-cycle,}
  \end{cases}
\end{gather*}
where the number $m(L)$ of smoothing steps on the finest level is a
free parameter. We
refer to $\B_L$ as the V-cycle preconditioner of $\A_L$. The
iteration
\begin{gather}
  \label{eq:6}
      \begin{pmatrix}
      u_{k+1}\\p_{k+1}
    \end{pmatrix}
    =
    \begin{pmatrix}
      u_{k}\\p_{k}
    \end{pmatrix}
    + \B_L \left(\L_L - \A_L
      \begin{pmatrix}
        u_{k}\\p_{k}
      \end{pmatrix}
    \right)
\end{gather}
is the V-cycle iteration.

The definition of the preconditioner $B_\ell: V_\ell \to V_\ell$ for the elliptic
operator $A_\ell$ follows the same concept, but dropping the pressure variables.

\subsection{Overlapping Schwarz smoothers}
\label{sec:smoother}

In this subsection, we define a class of smoothing operators $\mathcal
R_\ell$ based on a subspace decomposition of the space $X_\ell$.  Let
$\mathcal{N}_\ell$ be the set of vertices in the triangulation $\T_\ell$, and let
$\T_{\ell,\upsilon}$ be the set of cells in $\T_\ell$ sharing the
vertex $\upsilon$. They form a triangulation with $N(N>0)$ subdomains or
patches which we denote by $\{\Omega_{\ell,\upsilon}\}_{\upsilon = 1}^{N}$.

The subspace $X_{\ell,\upsilon} = V_{\ell,\upsilon} \times
Q_{\ell,\upsilon}$ consists of the functions in $X_\ell$ with support
in $\Omega_{\ell,\upsilon}$. Note that this implies homogeneous slip
boundary conditions on $\partial\Omega_{\ell,\upsilon}$ for the
velocity subspace $V_{\ell,\upsilon}$ and zero mean value on
$\Omega_{\ell,\upsilon}$ for the pressure subspace
$Q_{\ell,\upsilon}$. The Ritz projection $\mathcal P_{\ell,\upsilon}:
X_\ell \to X_{\ell,\upsilon}$ is defined by the equation
\begin{gather}
  \label{eq:Ritz-Stokes}
  \A_\ell\form(\mathcal P_{\ell,\upsilon}
  \begin{pmatrix}
    u_\ell\\p_\ell
  \end{pmatrix}
  ,
  \begin{pmatrix}
    v_{\ell,\upsilon}\\q_{\ell,\upsilon}  
  \end{pmatrix}
  )
  = \A_\ell\form(
  \begin{pmatrix}
    u_\ell\\p_\ell
  \end{pmatrix}
  ,
  \begin{pmatrix}
    v_{\ell,\upsilon}\\q_{\ell,\upsilon}
  \end{pmatrix}
  )
  \qquad \forall
  \begin{pmatrix}
    v_{\ell,\upsilon}\\ q_{\ell,\upsilon}
  \end{pmatrix}
  \in X_{\ell,\upsilon}.
\end{gather}
Note that each cell belongs to not more than four (eight in 3D)
patches $\T_{\ell,\upsilon}$, one for each of its vertices.

Then we define the additive Schwarz smoother
\begin{gather}
  \label{eq:smoother-Stokes}
  \mathcal{R}_\ell=
  \eta~\sum_{\substack{
      \upsilon \in \mathcal{N}_\ell
    }}
  \mathcal P_{\ell,\upsilon}\A^{-1}_\ell
\end{gather} 
where $\eta \in (0,1]$ is a scaling factor, $\mathcal R_\ell$ is $L^2$ symmetric and
positive definite.

Similarly, we define smoothers of the singularly perturbed elliptic
operator $A_\ell$, namely, $P_{\ell,\upsilon}: V_\ell \to
V_{\ell,\upsilon}$ is defined as
\begin{gather}
  \label{eq:Ritz-elliptic}
  A_\ell\form(P_{\ell,\upsilon} u_\ell,v_{\ell,\upsilon})
  = A_\ell\form(u_\ell,v_{\ell,\upsilon})
  \qquad \forall v_{\ell,\upsilon}\in V_{\ell,\upsilon},
\end{gather}
and the additive Schwarz smoother is

\begin{gather}
  \label{eq:smoother-elliptic}
  R_\ell=
  \eta~\sum_{\substack{
      \upsilon \in \mathcal{N}_\ell
    }}
  P_{\ell,\upsilon}A^{-1}_\ell.
\end{gather} 

%


\section{Convergence analysis}
\label{sec:analysis}

In this section, we provide a proof of the convergence for the
variable V-cycle iteration with additive Schwarz preconditioning
method. Our proof is based on the assumption that the domain $\Omega$
is bounded and convex, which will be omitted for simplicity in the
statement of following theorems and propositions. Our main result is:

\begin{theorem}
  \label{theorem:main}
  The multilevel iteration $\mathcal I - \mathcal B_L\mathcal
  A_L$ for the Stokes problem~\eqref{eq:Stokes-weak}
  with the variable V-cycle operator $\mathcal B_L$ defined in
  Section~\ref{sec:V-cycle} employing the smoother $\mathcal R_\ell$
  defined in equation~\eqref{eq:smoother-Stokes} with suitably small
  scaling factor $\eta$ is a contraction with contraction number
  independent of the mesh level $L$.
\end{theorem}

\begin{proof}
  First, we consider the nearly incompressible
  problem~\eqref{eq:nearly-incompressible}. For this weak formulation,
  we have by Theorem~\ref{theorem:equivalence}, that the multigrid
  method $\mathcal I - \mathcal B_{L,\epsilon} \mathcal A_{L,\epsilon}$ is equivalent to the
  method $I -  B_{L,\epsilon} A_{L,\epsilon}$ applied to the singularly perturbed
  problem~\eqref{eq:elliptic-2} in the velocity space.
  
  Convergence of the multilevel iteration $I - B_{L,\epsilon} A_{L,\epsilon}$ is shown in
  Theorem~\ref{theorem:singularly-perturbed} for all $\epsilon > 0$
  with a contraction number $\delta < 1$ independent of $L$ and
  $\epsilon$. Thus, by Theorem ~\ref{theorem:equivalence}, the same
  holds for $\mathcal I - \mathcal B_{L,\epsilon}\mathcal A_{L,\epsilon}$ with positive
  $\epsilon$.
  
  Finally, in~\eqref{eq:nearly-incompressible} we can let $\epsilon$
  converge to zero. The limit yields the well-posed Stokes
  problem~\eqref{eq:Stokes-weak}, and since the contraction number
  $\delta$ is independent of $\epsilon$, we obtain uniform convergence
  with respect to the mesh level $L$ in the limit $\epsilon\to 0$.
\end{proof}


The theorems and lemmas of the following subsections serve to
establish the building blocks of the proof of
Theorem~\ref{theorem:main}.

\subsection{The singularly perturbed problem}
\label{sec:singularly-perturbed}

%
\begin{theorem}
  \label{theorem:singularly-perturbed}
  Let $R_\ell$ be the smoother defined
  in~\eqref{eq:smoother-elliptic} with suitably small
  scaling factor $\eta$. Then,
  the multilevel iteration $ I -  B_L
  A_L$ with the variable V-cycle operator $ B_L$ defined in
  Section~\ref{sec:V-cycle} is a contraction with contraction number
  independent of the mesh level $L$ and the parameter $\epsilon$.
\end{theorem}

The proof of this theorem is postponed to page~\pageref{proof} and relies on
\begin{proposition}
  \label{proposition:singularly-perturbed}
 If $R_\ell$ satisfies the conditions:
  \begin{gather}
    \label{eq:2}
    A_L\bigl((I-R_\ell A_\ell) w,w\bigr) \geq 0,
    \quad \forall  w \in V_\ell
  \end{gather}
  and
  \begin{gather}
    \label{eq:1}
    (R^{-1}_\ell[I-P_{\ell-1}]w,[I-P_{\ell-1}]w)
    \leq \beta_\ell A_L([I-P_{\ell-1}]w,[I-P_{\ell-1}]w),
    \quad \forall  w \in V_\ell
  \end{gather} 
  where $\beta_\ell = O({\frac{1}{\gamma_\ell}})$ is defined in
  equation~\eqref{def:beta} below. Then
  \begin{gather}
    0 \leq A_L\bigl((I-B_\ell A_\ell)w,w\bigr) \leq \delta A_L(w,w),
    \quad \forall  w \in V_\ell
  \end{gather} 
  where $\delta = \frac{\hat{C}}{1+\hat{C}}$ and $\hat{C}$ are defined
  in Lemma~\ref{lemma:5}.
\end{proposition}

\begin{proof}
  In the case of self-adjoint operators $A_\ell$ which are inherited
  from a common bilinear form $a(.,.)$, this proposition would be part of
  the standard multigrid theory if $\beta_\ell$ were constant.
  Its proof can be adapted from similar theorems
  in~\cite{BraessHackbusch83,Bramble93,ArnoldFalkWinther97Hdiv}.
  We will prove the version needed here in the appendix.
\end{proof}
In the remainder of this section we use several propositions and
lemmas to establish our smoother $R_{\ell}$ satisfies the assumptions
of Proposition~\ref{proposition:singularly-perturbed}. For  $u\in
(I-P_{\ell-1})w$ with arbitrary $w\in V_\ell$, it follows from the discrete
Helmholtz decomposition in Section~\ref{sec:discretization} and the projection
operator $P_{\ell,\upsilon}$ in Section~\ref{sec:smoother} that $u$ admits a
local discrete Helmholtz decomposition
\begin{gather}
  \label{eq:local Helmholtz decomp}
  u_{\upsilon} = \naught{u}_{\upsilon} + u^{\perp}_{\upsilon}
\end{gather}

\begin{lemma}
  \label{lemma:equivalence-2}
  Given $L^2$-symmetric positive definite $R_{\ell}$ defined in
 ~\ref{sec:smoother} and symmetric positive definite $A_L(\cdot,\cdot)$
  defined in~\eqref{eq:elliptic-1} , there exists a constant $\eta \in
  (0,1]$ independent of $\ell$ such that
  \begin{gather}
    \eta (R_\ell^{-1}u,u) =
    \inf_{\substack{ u_{\upsilon}\in V_{\ell,\upsilon} \\
        \Sigma_{\upsilon}u_{\upsilon}=u
      }}
    \sum_{\substack{\upsilon \in \mathcal N_\ell
      }}
    A_L(u_{\upsilon},u_{\upsilon})
  \end{gather}
  \begin{proof}
    The following proof can be found in~\cite{ArnoldFalkWinther97Hdiv}
    for the $L^2$-inner product instead of $a_\ell(.,.)$. We
    copy it here to ascertain that it does not depend on the actual
    structure of the operator $A_L$ since it is purely algebraic. Thus, it
    applies to the operator $A_L$ in this paper as it applies to the
    different operator applied there.
    Recall that 
     \begin{gather}
      R_\ell =
      \eta~\sum_{\substack{
          \upsilon \in \mathcal{N}_\ell
        }}
      P_{\ell,\upsilon}A^{-1}_\ell =
       \eta~\sum_{\substack{
          \upsilon \in \mathcal{N}_\ell
        }}
        P_{\ell,\upsilon}A^{-1}_L.
    \end{gather}
    From
    \begin{gather}
      u=
      \sum_{\substack{
          \upsilon \in \mathcal{N}_\ell
        }}
      u_\upsilon
    \end{gather}
    we get 
    \begin{align}
      \eta (R_\ell^{-1}u,u) &= \eta \sum_{\substack{
          \upsilon \in \mathcal{N}_\ell
        }}
      (R_\ell^{-1}u,u_\upsilon)  \\
      &= \eta \sum_{\substack{
          \upsilon \in \mathcal{N}_\ell
        }}
      (A_L P_{\ell,\upsilon}A_L^{-1}R_\ell^{-1}u,u_\upsilon) \\
        &\leq  \eta^{\frac{1}{2}}\left\{\sum_{\substack{
            \upsilon \in \mathcal{N}_\ell
          }}
        (A_L \eta P_{\ell,\upsilon}A_L ^{-1}R_\ell^{-1}u, A_L^{-1}R_\ell^{-1}u)\right\}^{\frac{1}{2}} 
      \left\{ \sum_{\substack{
            \upsilon \in \mathcal{N}_\ell
          }}
        (A_L u_\upsilon,u_\upsilon)\right\}^{\frac{1}{2}} \\
      &= \eta^{\frac{1}{2}}\left\{ (A_L u, A_L^{-1}R^{-1}_{\ell}u)\right\}^{\frac{1}{2}} \left\{\sum_{\substack{
            \upsilon \in \mathcal{N}_\ell
          }} 
        (A_L u_\upsilon,u_\upsilon) \right\}^{\frac{1}{2}} \\
      &=\eta^{\frac{1}{2}} \left\{ (u, R^{-1}_{\ell}u)\right\}^{\frac{1}{2}} \left\{\sum_{\substack{
            \upsilon \in \mathcal{N}_\ell
          }} 
        (A_L u_\upsilon,u_\upsilon) \right\}^{\frac{1}{2}} 
    \end{align}
    The above inequality works for arbitrary splitting, hence we have 
    \begin{gather}
      \eta (R_\ell^{-1}u,u) \leq
      \sum_{\substack{\upsilon \in \mathcal N_\ell
        }}
      A_L(u_{\upsilon},u_{\upsilon})
    \end{gather}
    For the choice $u_{\upsilon}=P_{\ell,\upsilon}P_{\ell}A_L^{-1}R^{-1}u$ we get
    \begin{gather}
      \eta (R_\ell^{-1}u,u) =
      \inf_{\substack{ u_{\upsilon}\in V_{\ell,\upsilon} \\
          \Sigma_{\upsilon}u_{\upsilon}=u
        }}
      \sum_{\substack{\upsilon \in \mathcal N_\ell
        }}
      A_L(u_{\upsilon},u_{\upsilon})
    \end{gather}
  \end{proof}
\end{lemma}

\begin{lemma}
  \label{lemma:estimate-2}
  Given the local Helmholtz decomposition in~\eqref{eq:local Helmholtz decomp}. For any
  $u^{\perp}_{\upsilon}\in V_{\ell,\upsilon}$ , there exists constant
  $C_1$ independent of multigrid level satisfying:
  \begin{gather}
    \sum_{\upsilon \in \mathcal{N}_\ell}
    \norm{\div u^{\perp}_{\upsilon}}^{2}
    \leq
    C_1 \sum_{\upsilon \in \mathcal{N}_\ell}
    a_\ell(u^{\perp}_{\upsilon}, u^{\perp}_{\upsilon}) 
  \end{gather}
\end{lemma}
\begin{proof}
  It follows from Lemma~\ref{lemma:estimate-1} that the estimate $\norm{\div u^{\perp}}^{2}
  \leq C a_\ell(u^{\perp}, u^{\perp}) $ hold for all $u^{\perp}\in
  V_\ell$. It is easy to see that $V_{\ell,\upsilon}$ is a subspace of
  $V_{\ell}$ for any $\upsilon$, so the estimate are also valid on any
  patch. In 2-D case, one cell could at most be sharing by four
  patches(eight patches in 3D). Hence there exists a constant $C_1$
  independent of multigrid level such that the estimates holds for the
  summation of local estimates.
\end{proof}

\begin{proposition}
  \label{proposition:estimate-3}
  Given the overlapping subspace decomposition of $V_\ell$
  in~\ref{sec:smoother} and the interior penalty bilinear form
  $a_{\ell}(u, v)$ in~\eqref{eq:IP}.  Assume $\sigma_\ell$ is chosen
  sufficiently large, the following estimate holds on each level
  $\ell$.  Then, there is a constant $C_2$ which is independent of
  multigrid level such that for any $u \in V_\ell$ holds
  \begin{gather}
    \sum_{\upsilon \in \mathcal{N}_\ell}
    a_{\ell}(u_{\upsilon}, u_{\upsilon})
    \leq
    C_2a_{\ell}(u,u)
  \end{gather}
\end{proposition}
\begin{proof}
  For a fixed $L$, the penalty constant $\sigma_\ell$ is
  ${\overline{\sigma}}/{h_L}$ which is greater than
  ${\overline{\sigma}}/{h_\ell}$. For the latter, this is a standard
  result: the proof and details on the choice of $\overline{\sigma}$
  can be found in~\cite[p. 1361]{FengKarakashian01}.
\end{proof}

\begin{proof}[Proof of Theorem~\ref{theorem:singularly-perturbed}]
  \label{proof}
  Recall the definition of $A_{L}$-orthogonal projections $P_\ell$ and
  $P_{\ell,\nu}$ which restrict the projection on
  $\Omega_{\ell,\nu}$(zero
  elsewhere). Following~\cite{ArnoldFalkWinther97Hdiv}, we show that
  if $0<\eta\leq 1/4$ , the smoother $R_{\ell}$ satisfies the first
  condition in Theorem~\ref{theorem:singularly-perturbed}.
  
   For  $w\in V_\ell$
  \begin{gather}
    A_L([I-R_\ell A_\ell]w,w) = A_L(w,w) - \eta \sum_{ \upsilon \in
      \mathcal N_\ell}A_L(P_{\ell,\upsilon}w,w)  
  \end{gather} 
  but 
  \begin{gather}
    A_L(P_{\ell,\upsilon}w,w) =
    A_L(P_{\ell,\upsilon}w,P_{\ell,\upsilon}w)
    \leq  A_L(w,w)^{\frac{1}{2}} A_L(P_{\ell,\upsilon}w,P_{\ell,\upsilon}w)^{\frac{1}{2}}
  \end{gather} 
  so
  \begin{gather}
    \sum_{\substack{
        \upsilon \in \mathcal{N}_\ell
      }}
    A_L(P_{\ell,\upsilon}w,w) 
    \leq 
    \sum_{\substack{
        \upsilon \in \mathcal{N}_\ell
      }}
   A_L(w,w)
    \leq
    4A_L(w,w)
  \end{gather}
  Hence the first hypothesis holds.

  Thus, it remains to check the second condition which could be
  reduced to the following problem: 
  for $u = (I - P_{\ell-1})w$ (where $w \in V_\ell$) with the decomposition $ u = \sum_{\upsilon}u_{\upsilon}$, there is a constant $C$ such that
  \begin{gather}
  \label{eq:stable decomp}
    \sum_{\upsilon \in \mathcal{N}_\ell}
    A_L(u_{\upsilon},u_{\upsilon}) 
    \leq
    C A_L(u,u)
  \end{gather}

  Following Lemmas~\ref{lemma:estimate-1}, \ref{lemma:equivalence-2},
  \ref{lemma:estimate-2} and Proposition \ref{proposition:estimate-3},
  we get:
  \begin{align}
    \sum_{\upsilon \in \mathcal{N}_\ell}
    A_L(u_{\upsilon},u_{\upsilon})
    &= \sum_{\upsilon \in \mathcal{N}_\ell}
    \left\{ a_{\ell}(u_{\upsilon}, u_{\upsilon})+\epsilon^{-1}(\nabla \cdot u_{\upsilon}, \nabla \cdot  v_{\upsilon})\right\}  \\
    &=
    \sum_{\substack{
        \upsilon \in \mathcal{N}_\ell
      }}
    \left\{ a_{\ell}(u_{\upsilon}, u_{\upsilon})+\epsilon^{-1}(\nabla \cdot u^{\perp}_{\upsilon}, \nabla \cdot  u^{\perp}_{\upsilon})\right\}  \\
    &\leq
    C_2a_{\ell}(u,u) +\sum_{\substack{
        \upsilon \in \mathcal{N}_\ell
      }}
    C_1\frac1\alpha \epsilon^{-1}a_{\ell}(u^{\perp}_{\upsilon},u^{\perp}_{\upsilon})  \\ 
    &\leq
    C_2a_{\ell}(u,u)+  \epsilon^{-1}C_1 \frac1\alpha a_{\ell}( u^{\perp}, u^{\perp}) \\
    &\leq
    C_2a_{\ell}(u,u)+  \epsilon^{-1}C_1\frac1\alpha\frac1\gamma_\ell (\div u^{\perp}, \div u^{\perp}) \\
    &=
    C_2a_{\ell}(u,u)+  \epsilon^{-1}C_1\frac1\alpha\frac1\gamma_\ell (\div u, \div u) \\
    &\leq
    \max\left\{C_2,C_1\frac1\alpha\frac1\gamma_\ell\right\}A_L(u,u) \\
    &=
    C_{\ell}A_L(u,u)
  \end{align} 
  where $C_{\ell} = \max \left\{C_2,C_1\frac1\alpha\frac1\gamma_\ell\right\}$.
  \end{proof}
Now set 
\begin{gather}
\label{def:beta}
\beta_\ell = \frac{1}{\eta}C_{\ell}
\end{gather}
We have verified the two conditions in Proposition ~\ref{proposition:singularly-perturbed}.

\begin{lemma}
\label{lemma:5}
Given $\beta_\ell$ above and $m(\ell)$ defined in ~\ref{sec:V-cycle}, there is a constant $\hat{C}$ such that 
\begin{gather}
\frac{\beta_\ell}{2m(\ell)} \leq \hat{C}
\end{gather}
\end{lemma}
\begin{proof}
We will discuss this inequality in two cases: first, if $\beta_\ell = \frac{1}{\eta}C_2$, then 
\begin{gather}
\frac{\beta_\ell}{2m(\ell)} = \frac{\frac{1}{\eta}C_2}{2m_0 2^{L-\ell}} \leq  \frac{\frac{1}{\eta}C_2}{2m_0} =: \hat{C}
\end{gather}
On the other hand, if $\beta = \frac{1}{\eta}C_1\frac1\alpha\frac1\gamma$
\begin{gather}
\frac{\beta_\ell}{2m(\ell)} = \frac{C_1\frac{1}{\alpha}\frac{1}{c}\sqrt{2^{L-\ell}}}{2m_0 2^{L-\ell}} = \frac{C_1\frac{1}{\alpha}\frac{1}{c}}{2m_0 \sqrt{2^{L-\ell}}}  \leq \frac{C_1\frac{1}{\alpha}\frac{1}{c}}{2m_0}=: \hat{C}
\end{gather}
\end{proof}

\subsection{The mixed problem}
\label{sec:mixed-problem}

Secondly, we will discuss the Stokes equation in mixed
variables. Set $X_{\ell,\epsilon} := \{(u_\ell,p_\ell) \in X_\ell : \div u_\ell = \epsilon p_\ell \}$. Now, it remains to show the equivalence between the multigrid algorithms.

\begin{proposition}
  The multigrid components fulfill the following properties:
  \begin{enumerate}
  \item The smoother $\mathcal{R}_\ell$ for the mixed problem defined
    in~\eqref{eq:smoother-Stokes} preserves $X_{\ell,\epsilon}$. On
    the subspace it is equivalent to the smoother $R_\ell$ in primal
    variables. This means for $(u_\ell, p_\ell) \in X_{\ell,\epsilon}$ and
    \begin{gather}
      \begin{pmatrix}
        \hat{u}_\ell\\\hat{p}_\ell
      \end{pmatrix}
      = \mathcal{R}_\ell
      \begin{pmatrix}
        u_\ell\\p_\ell
      \end{pmatrix}
    \end{gather}
    there holds $(\hat{u_\ell}, \hat{p_\ell}) \in X_{\ell,\epsilon}$ and 
    \begin{gather}
      \hat{u}_\ell = R_\ell u_\ell
    \end{gather}

  \item The prolongation $\I_{\ell-1}$ maps $X_{\ell-1,\epsilon}$ into
    $X_{\ell,\epsilon}$. On the subspace it is equivalent to the prolongation
    $I_\ell$ in primal variables. This means for $(u_{\ell-1},
    p_{\ell-1}) \in X_{\ell-1,\epsilon}$ and
    \begin{gather}
      \begin{pmatrix}
        \hat{u}_\ell\\\hat{p}_\ell
      \end{pmatrix}
      = \I_\ell
      \begin{pmatrix}
        u_{\ell-1}\\p_{\ell-1}
      \end{pmatrix}
    \end{gather}
    there holds $(\hat{u}_\ell, \hat{p}_\ell) \in X_{\ell,\epsilon}$ and 
    \begin{gather}
      \hat{u}_\ell = I_{\ell}(u_{\ell-1})
    \end{gather}

  \item The coarse grid solution operator maps $X_{\ell-1,\epsilon}$ into
    $X_{\ell,\epsilon}$. On the subspace it is equivalent to the coarse grid
    solution operator in primal variables. This means for $(u_\ell,
    p_\ell) \in X_{\ell,\epsilon}$ and
    \begin{gather}
      \begin{pmatrix}
        \hat{u}_{\ell-1}\\\hat{p}_{\ell-1}  
      \end{pmatrix}
      = \A^{-1}_{\ell-1}[\I_{\ell-1}]^t\A_\ell
      \begin{pmatrix}
        u_{\ell}\\p_{\ell}
      \end{pmatrix}
    \end{gather}
    there holds $(\hat{u}_{\ell-1}, \hat{p}_{\ell-1}) \in X_{\ell-1,\epsilon}$ and 
    \begin{gather}
      \hat{u}_{\ell-1} = A_{\ell-1}^{-1}[I_{\ell-1}]^t A_\ell u_\ell
    \end{gather}
  \end{enumerate}
\end{proposition}

\begin{proof}
  The proof of this proposition can be found for the operators there
  in~\cite[p. 93]{Schoeberl99dissertation}. We do not provide it here
  since the arguments are purely linear algebra, and thus apply
  independent of the actual bilinear form.
\end{proof}

\begin{theorem}
  \label{theorem:equivalence}
  The multigrid algorithm in mixed variables preserves the space
  $X_{\ell,\epsilon}$. On this subspace it is equivalent to the
  multigrid algorithm in primal variables. This means for $(u_\ell,
  p_\ell) \in X_{\ell,\epsilon}$ and $(\hat{u}_\ell, \hat{p}_\ell) =
  \B_\ell(u_\ell,p_\ell)$ there holds $(\hat{u}_\ell, \hat{p}_\ell)
  \in X_{\ell,\epsilon}$ and
  \begin{gather}
    \hat{u_\ell} = B_\ell u_\ell
  \end{gather}
  where $ \B_\ell$ and $ B_\ell$ are the corresponding multigrid operators for each algorithm. 
\end{theorem}
\begin{proof}
  The multigrid operator $B_\ell$ fulfills the recursion 
  \begin{gather}
    B_0 = A_0^{-1},
  \end{gather}
  
  \begin{gather}
    B_\ell = (R_\ell)^{m_\ell}(I - I_\ell(I - (B_{\ell-1}))A_{\ell-1}^{-1}[I_{\ell-1}]^t A_\ell)(R_\ell)^{m_\ell},
  \end{gather}
  and the mixed operator $\B_\ell$ fulfills a corresponding
  one. Then we apply the above proposition, and the theorem is proved
  by induction.
\end{proof}

\section{Numerical results}
\label{sec:experiments}

We test the additive Schwarz method which we have analyzed in the
preceding section in order to ascertain that the contraction numbers
are not only bounded away from one, but are actually small enough to
make this method interesting. Furthermore, we conduct experiments,
which go beyond our analysis, in particular regarding the choice of
the penalty parameter and the number of smoothing steps on lower
levels.

The experimental setup for most of the tables is as follows: the
domain is $\Omega = [-1,1]^2$, the coarsest mesh $\T_0$ consists of a
single cell $T=\Omega$. The mesh $\T_\ell$ on level $\ell$ is obtained
by dividing all cells in $\T_{\ell-1}$ into four quadrilaterals by
connecting the edge midpoints. Thus, a mesh on level $\ell$ has
$4^\ell$ cells, and the length of their edges is $2^{1-\ell}$.  The
right hand side is $f=(1,1)$. For the relaxation parameter in the
additive Schwarz method, we found that $0.5$ is the value which
provides the best results for all experimental setups, hence we keep
it there in all the following experimental setups.



In Table ~\ref{table:1}, we first test the additive Schwarz smoother
using variable V-cycle algorithm on a square domain with no-slip
boundary condition. For the penalty constant in the DG form
~\eqref{eq:IP}, we choose the penalty parameter as
$\bar{\sigma}/h_{L}$, where $\bar{\sigma}=(k+1)(k+2)$, on the
finest level $L$ and all lower levels $\ell$. Results for different
pairs of $RT_k/Q_k$ are reported in the table which show the fast and uniform
convergence.

\begin{table}[tp]
  \center
  \begin{tabular}{ | c | c c |}
    \hline
    level&$RT_1$&$RT_2$\\ \hline
    3&4&4
    \\ \hline
    4&4&4
    \\ \hline
    5&4&4
    \\ \hline
    6&4&4
    \\   \hline
    7&4&4
    \\ \hline
    8&4&5
      \\ \hline  
  \end{tabular}
  \caption{Number of iterations $n_{6}$ to reduce the residual by $10^{-6}$ with the
    variable V-cycle algorithm with penalty parameter dependent of the finest level mesh size.}
     \label{table:1}
 \end{table}

 In Table ~\ref{table:2}, we keep the same experimental setup and
 present iteration counts for the standard V-cycle algorithm with one
 and two pre- and post-smoothing steps, respectively. Although our
 analysis does not apply, we still observe uniform convergence
 results. We also see that the variable V-cycle with a single
 smoothing step on the finest level is as fast as the standard V-cycle
 with two smoothing steps, and thus the variable V-cycle is more
 efficient.

 \begin{table}[tp]
  \center
  \begin{tabular}{ | c | c c | c c |}
    \hline
    & \multicolumn{2}{c|}{$m(\ell) = 1$}
    & \multicolumn{2}{c|}{$m(\ell) = 2$}
    \\\hline
    level&$RT_1$&$RT_2$&$RT_1$&$RT_2$ \\ \hline
    3&7&7&4&4
    \\ \hline
    4&7&7&4&4
    \\ \hline
    5&7&7&4&4
    \\ \hline
    6&7&7&4&4
    \\   \hline
    7&8&8&4&4
    \\ \hline
    8&8&8&4&4
      \\ \hline  
  \end{tabular}
  \caption{Number of iterations $n_{6}$ to reduce the residual by
    $10^{-6}$ with the standard V-cycle iteration with one and two
    pre- and post-smoothing steps. Penalty
    parameter dependent of the finest level mesh size.}
     \label{table:2}
 \end{table}

 In Table ~\ref{table:4}, we test the variable and standard V-cycles
 with penalty parameters depending on the mesh level $\ell$, namely
 $\bar{\sigma}/h_{\ell}$ (where $\bar{\sigma}$ is a positive constant
 depending on the polynomial degree) in the DG form
 ~\eqref{eq:IP}. While our convergence analysis does not cover this
 case either, we observe convergence rates equal to the case with
 inherited forms.

\begin{table}[tp]
  \center
  \begin{tabular}{ | c | c c | c c | }
    \hline
    & \multicolumn{2}{c|}{variable}
    & \multicolumn{2}{c|}{standard}
    \\\hline
    level&$RT_1$&$RT_2$&$RT_1$&$RT_2$\\ \hline
    3&4&4&7&7
    \\ \hline
    4&4&4&7&7
    \\ \hline
    5&4&4&7&7
    \\ \hline
    6&4&4&7&7
    \\   \hline
    7&4&4&7&8
    \\ \hline
    8&4&5&8&8
      \\ \hline  
  \end{tabular}
  \caption{ Penalty parameter dependent on the mesh size of each level. Number of iterations $n_{6}$ to reduce
    the residual by $10^{-6}$ with variable and standard V-cycle
    iterations with $m(L) = 1$.}
    \label{table:4}
 \end{table}

 In Table ~\ref{table:6}, we provide results with GMRES solver and
 $\B_L$ as preconditioner for different experimental setups as in
 Tables~\ref{table:1},~\ref{table:2} and ~\ref{table:4}
 respectively. The second and third columns are results for variable
 V-cycle with penalty parameter dependent of the finest level mesh
 size. The fourth and fifth columns are the results
 for standard V-cycle with penalty parameter dependent of the finest
 level mesh size. The last two columns are the results for standard
 V-cycle with penalty parameter depend on the mesh size of each
 level. From this table, we see that the GMRES method, as expected, is
 faster in every case.

 \begin{table}[tp]
  \center
  \begin{tabular}{ | c | c c | c c | c c |}
    \hline
    & \multicolumn{2}{c|}{variable}
    & \multicolumn{2}{c|}{standard}
    & \multicolumn{2}{c|}{noninherited}
    \\\hline
    level&$RT_1$&$RT_2$&$RT_1$&$RT_2$&$RT_1$&$RT_2$\\ \hline
    3&2&2&2&2&2&2
    \\ \hline
    4&3&3&3&3&3&3
    \\ \hline
    5&3&3&4&3&4&4
    \\ \hline
    6&3&3&5&4&5&5
    \\   \hline
    7&3&3&5&5&5&5
    \\ \hline
    8&5&4&6&6&8&6
      \\ \hline  
  \end{tabular}
   \caption{Number of iterations $n_{6}$ to reduce
    the residual by $10^{-6}$ with GMRES solver and preconditioner
    $\B_L$; variable and standard V-cycle with inherited forms,
    variable V-cycle with noninherited forms. One pre- and
    post-smoothing step on the finest level.}
    \label{table:6}
 \end{table}

\section{Conclusions}
In this paper, we have investigated smoothers based on the ones
introduced by Arnold, Falk, and Winther for problems in $\Hdiv$ in a
variable V-cycle preconditioner for the Stokes system.  We presented the
convergence analysis and showed uniform contraction independent on the
mesh level. In numerical experiments we showed that the contraction
is not only uniform, but also very fast, thus making our method a
feasible solver or preconditioner.

In theory, the performance of the smoother relies on an exact sequence
property of finite element spaces, in particular an $\Hdiv$-conforming
discontinuous Galerkin discretization of the Stokes problem.
Our experiments with the Taylor--Hood elements, where the method fails,
demonstrate that this is not an artifact of the analysis, but that the
technique does not work due to the lack of an exact
Hodge decomposition and nested divergence free subspaces.

\appendix
\section{Proof for Proposition~\ref{proposition:singularly-perturbed}}

Following the proof in ~\cite{ArnoldFalkWinther97Hdiv}, we want to
show by induction on $i$ that
\begin{gather} 
0\leq A_{L}((I-B_i A_i)u,u)\leq \delta A_L(u,u), \quad \forall u\in V_\ell
\end{gather}
For $i=1$ is obvious since $B_1=A_1^{-1}$. Now check if the above
inequality hold for $i=\ell -1$. Recall the relaxation operator
$K_\ell = I - R_\ell A_\ell$ and the recurrence relation :
\begin{gather}
  I-B_\ell A_\ell =
  K^{m(\ell)}_{\ell} \bigl[
  (I - P_{\ell -1})+(I - B_{\ell -1}A_{\ell -1})P_{\ell -1}
  \bigr] K^{m(\ell)}_{\ell}
\end{gather}
The lower bound easily follows from the inductive hypothesis and the
above identity. For the upper bound, we use the induction hypothesis
to obtain

\begin{align}
  \label{inequality:1}
  A_{L}((I-B_\ell A_\ell)u,u)
  &\leq A_L([I - P_{\ell -1}]K^{m(\ell)}_{\ell}u,K^{m(\ell)}_{\ell}u)
  + \delta A_L({P_{j-1}K^{m(\ell)}_{\ell}u,K^{m(\ell)}_{\ell}u})
  \\
  &= (1-\delta)A_L([I - P_{\ell
    -1}]K^{m(\ell)}_{\ell}u,K^{m(\ell)}_{\ell}u)
  +\delta A_L({K^{m(\ell)}_{\ell}u,K^{m(\ell)}_{\ell}u}).
\end{align}

Now by the orthogonality from~\eqref{eq: ritz projection}

\begin{align}
 &A_L([I - P_{\ell -1}]K^{m(\ell)}_{\ell}u,[I - P_{\ell -1}]K^{m(\ell)}_{\ell}u) \\
 &= A_L([I - P_{\ell -1}]K^{m(\ell)}_{\ell}u,K^{m(\ell)}_{\ell}u) \\
 &= ([I - P_{\ell -1}]K^{m(\ell)}_{\ell}u,A_{\ell}K^{m(\ell)}_{\ell}u) \\
 &=  (R_{\ell}^{-1}[I - P_{\ell -1}]K^{m(\ell)}_{\ell}u,R_{\ell}A_{\ell}K^{m(\ell)}_{\ell}u) \\
 &\leq (R_{\ell}^{-1}[I - P_{\ell -1}]K^{m(\ell)}_{\ell}u,[I - P_{\ell -1}]K^{m(\ell)}_{\ell}u)^{\frac{1}{2}}(R_{\ell}A_{\ell}K^{m(\ell)}_{\ell}u,A_{\ell}K^{m(\ell)}_{\ell}u)^{\frac{1}{2}} \\
&\leq \sqrt{\beta_\ell} ([I - P_{\ell -1}]K^{m(\ell)}_{\ell}u,[I - P_{\ell -1}]K^{m(\ell)}_{\ell}u)^{\frac{1}{2}}(R_{\ell}A_{\ell}K^{m(\ell)}_{\ell}u,A_{\ell}K^{m(\ell)}_{\ell}u)^{\frac{1}{2}} 
\end{align}
Hence, we get 
\begin{align}
&A_L([I - P_{\ell -1}]K^{m(\ell)}_{\ell}u,K^{m(\ell)}_{\ell}u) \leq \beta_\ell (R_{\ell}A_{\ell}K^{m(\ell)}_{\ell}u,A_{\ell}K^{m(\ell)}_{\ell}u) \\
&= \beta_\ell A_L([I - K_{\ell -1}]K^{2m(\ell)}_{\ell}u,u)
\end{align}

It follows from the positive semi-definiteness and ~\eqref{eq:2} that
the spectrum of $K_\ell$ is contained in the interval
$[0,1]$. Therefore, we have
\begin{gather}
\label{inequality:2}
A_L([I - K_{\ell -1}]K^{2m(\ell)}_{\ell}u,u) \leq A_L([I - K_{\ell -1}]K^{i}_{\ell}u,u),  \quad for \quad i \leq 2m(\ell)
\end{gather}
whence
\begin{align}
&A_L([I - K_{\ell -1}]K^{2m(\ell)}_{\ell}u,u) \leq \frac{1}{2m(\ell)} \sum_{i=0}^{2m(\ell)-1}A_L([I-K_{\ell}]K_{\ell}^{i}u,u) \\
&= \frac{1}{2m(\ell)} A_L([I-K_{\ell}]K_{\ell}^{2m(\ell)}u,u) 
\end{align}

Combining ~\eqref{inequality:1} and ~\eqref{inequality:2} and
following Lemma ~\ref{lemma:5}, we get
\begin{align}
&A_{L}((I-B_\ell A_\ell)u,u)  \leq (1-\delta)\frac{\beta_\ell}{2m(\ell)}A_L([I - K^{2m(\ell)}]u,u) +\delta A_L({K^{m(\ell)}_{\ell}u,K^{m(\ell)}_{\ell}u})  \\
&\leq (1-\delta)\hat{C}A_L([I - K^{2m(\ell)}]u,u) +\delta A_L({K^{m(\ell)}_{\ell}u,K^{m(\ell)}_{\ell}u})  \\
&= (1-\delta)\hat{C}A_L(u,u) +[\delta-(1-\delta)\hat{C}] A_L({K^{m(\ell)}_{\ell}u,K^{m(\ell)}_{\ell}u}) 
\end{align}
The results now follows by choosing :
\begin{gather}
  \delta = (1-\delta)\hat{C},
  \quad\text{i. e.,}\quad
  \delta = \frac{\hat{C}}{1+\hat{C}}
\end{gather}

\section*{Acknowledgments}
Numerical results in this publication were produced using the deal.II
library~\cite{dealII81} described
in~\cite{BangerthHartmannKanschat07}. We would like to
thank Joachim Schöberl for interesting and helpful discussion about
his work
in~\cite{Schoeberl98,Schoeberl99dissertation,SchoeberlZulehner03}.

\bibliographystyle{abbrv}
\bibliography{all,kanschat,local}

\end{document}